\documentclass[10pt]{amsart}

\usepackage[pagebackref]{hyperref}
\hypersetup{
    colorlinks = true,
    citecolor = [rgb]{0.00, 0.5, 0.00},
}

\renewcommand*{\backref}[1]{}
\renewcommand*{\backrefalt}[4]{%
    \ifcase #1 (Not cited.)%
    \or        (Cited on page~#2.)%
    \else      (Cited on pages~#2.)%
    \fi}

\usepackage{mathtools}

\usepackage[margin=1.375in]{geometry}
\usepackage{todonotes}
\usepackage{multirow}
\usepackage{longtable}

\usepackage{amsmath}
\usepackage{amsfonts}
\usepackage{amssymb}
\usepackage{float}
\usepackage{amsthm}
\usepackage{comment}
\usepackage{amscd}
\usepackage{amsxtra}
\usepackage{epsfig}

\usepackage{listings}
\usepackage{color}

\definecolor{dkgreen}{rgb}{0,0.6,0}
\definecolor{gray}{rgb}{0.5,0.5,0.5}
\definecolor{mauve}{rgb}{0.58,0,0.82}

\usepackage{color}
\usepackage[all]{xy}
\usepackage{verbatim}
\usepackage{setspace}

\usepackage{eucal}
\usepackage{mathrsfs}

\usepackage{graphicx}
\usepackage{tikz-cd}

\usepackage{url}
\usepackage{verbatim}

\usepackage{xy}
\xyoption{all}
\newcommand{\xar}[1]{\xrightarrow{{#1}}}

\usepackage{amsthm}

\usepackage{caption} 
\makeatletter
\def\@tocline#1#2#3#4#5#6#7{\relax
  \ifnum #1>\c@tocdepth 
  \else
    \par \addpenalty\@secpenalty\addvspace{#2}%
    \begingroup \hyphenpenalty\@M
    \@ifempty{#4}{%
      \@tempdima\csname r@tocindent\number#1\endcsname\relax
    }{%
      \@tempdima#4\relax
    }%
    \parindent\z@ \leftskip#3\relax \advance\leftskip\@tempdima\relax
    \rightskip\@pnumwidth plus4em \parfillskip-\@pnumwidth
    #5\leavevmode\hskip-\@tempdima
      \ifcase #1
       \or\or \hskip 1em \or \hskip 2em \else \hskip 3em \fi%
      #6\nobreak\relax
    \hfill\hbox to\@pnumwidth{\@tocpagenum{#7}}\par
    \nobreak
    \endgroup
  \fi}
\makeatother

\DeclareMathOperator{\spf}{Spf}

\DeclareMathOperator{\tr}{Tr}

\newtheorem{lemma}{Lemma}[section]

\newtheorem{theorem}[lemma]{Theorem}

\newtheorem{prop}[lemma]{Proposition}

\theoremstyle{definition}

\newtheorem{question}[lemma]{Question}

\newtheorem{definition}[lemma]{Definition}

\newtheorem{example}[lemma]{Example}

\newtheorem{remark}[lemma]{Remark}

\newcommand{\FF}{\mathbf{F}}
\newcommand{\Z}{\mathbf{Z}}

\newcommand{\Lk}{L_{K(n)}}

\newcommand{\M}{{M}}

\newcommand{\Mfg}{\M_{FG}}

\newcommand{\co}{\mathcal{O}}

\renewcommand{\SS}{\mathbf{S}}
\newcommand{\GG}{\mathbb{G}}

\newcommand{\Lone}{L_{K(1)}}

\newcommand{\Eoo}{{\mathbf{E}_\infty}}
\newcommand{\EOo}{{\mathbf{E}_\infty}}

\newcommand{\ol}[1]{\overline{#1}}

\newcommand{\fr}[1]{\mathfrak{#1}}

\newcommand{\E}[1]{\mathbf{E}_{{#1}}}

\newcommand{\bu}{\mathrm{bu}}
\newcommand{\bo}{\mathrm{bo}}

\newcommand{\TMF}{\mathrm{TMF}}

\newcommand{\W}{\mathrm{W}}

\title{Roots of unity in $K(n)$-local rings}
\author{Sanath Devalapurkar}
\email{sanathd@mit.edu}
\subjclass[2010]{Primary 55P43, 55S12}

\begin{document}
\begin{abstract}
    The goal of this paper is to address the following question: if $A$ is an
    $\E{k}$-ring for some $k\geq 1$ and $f\colon\pi_0 A \to B$ is a map of
    commutative rings, when can we find an $\E{k}$-ring $R$ with an $\E{k}$-ring
    map $g\colon A \to R$ such that $\pi_0 g = f$?  A classical result in the
    theory of realizing $\Eoo$-rings, due to Goerss--Hopkins, gives an
    affirmative answer to this question if $f$ is \'etale. The goal of this
    paper is to provide answers to this question when $f$ is ramified. We prove
    a non-realizability result in the $K(n)$-local setting for every $n\geq 1$
    for $H_\infty$-rings containing primitive $p$th roots of unity. As an
    application, we give a proof of the folk result that the Lubin--Tate tower
    from arithmetic geometry does not lift to a tower of $H_\infty$-rings over
    Morava $E$-theory.
\end{abstract}
\maketitle

\section{Introduction}\label{intro}
In this paper, we study some rigidity properties of structured ring spectra in
$K(n)$-local spectra (see \cite{hs} for a thorough introduction). The primary
question we are concerned with in this paper is the following:
\begin{question}\label{q1}
    Let $A$ be an $\E{k}$-ring for some $k\geq 1$, and let $f\colon\pi_0 A \to
    B$ be a map of commutative rings. Can we find an $\E{k}$-ring $R$ with an
    $\E{k}$-ring map $g\colon A \to R$ such that $\pi_0 g = f$?
\end{question}
As one might expect, Question \ref{q1} is hard to resolve in general,
particularly without any restrictions on the map $f$. Variants of this question
with additional conditions imposed on the map $f$ have been studied quite
extensively. In the case $k=\infty$, Question \ref{q1} was originally raised in
\cite{schwanzl-vogt}, where they used topological Hochschild homology to provide
a negative answer to the general form of the question: they showed that the
extension $\Z\to \Z[i]$ does not lift to an $\Eoo$-ring extension of the sphere
spectrum. The question of adjoining a primitive root of unity to an $\Eoo$-ring
in the {un}ramified case was studied in \cite[Theorem 3]{schwanzl-vogt} and
\cite[Example 2.2.8]{baker-richter-realize}. We refer the reader to the lovely
paper \cite{goerss-realizing} for further discussion of Question \ref{q1} from
the perspective of chromatic homotopy theory.

One solution to Question \ref{q1} (which sets up a general framework for
answering such questions) is provided by Goerss--Hopkins obstruction theory (see
\cite{goerss-hopkins}); using their work, the following result was deduced in
\cite{baker-richter-realize} (see also \cite[Theorem 7.5.0.6]{HA}):
\begin{theorem}\label{affirmative1}
    Question \ref{q1} admits an affirmative answer if $f$ is an \'etale map.
\end{theorem}
For instance, if $k$ is a finite field of characteristic $p$ and $\W(k)$ is the
ring of Witt vectors, then the map $\Z_p \to \W(k)$ is \'etale. As $\pi_0 \SS_p
= \Z_p$, Theorem \ref{affirmative1} gives an $\Eoo$-ring $\SS_{\W(k)}$, called
the spherical Witt vectors, such that $\pi_0 \SS_{\W(k)} = \W(k)$.

The goal of this paper is to address Question \ref{q1} in the world of
$K(n)$-local spectra when $f$ is ramified. Our main result states that $f$ in
general cannot be lifted to a map of structured ring spectra, even if one
restricts to studying $H_\infty$-rings. A $H_\infty$-ring structure is the
analogue of an $\Eoo$-ring structure in the stable \emph{homotopy} category; in
particular, $\Eoo$-rings provide examples of $H_\infty$-rings. We show:
\begin{theorem}\label{failure2}
    Let $n>0$. There is no nontrivial $K(n)$-local $H_\infty$-ring $R$ at a
    prime $p>2$ (resp. at $p=2$) such that $\pi_0 R$ contains a primitive $p$th
    root of unity (resp. a $4$th root of unity).
\end{theorem}
Since $\Eoo$-rings are, in particular, $H_\infty$-rings, Theorem \ref{failure2}
shows that no nontrivial $K(n)$-local $\Eoo$-ring $R$ at a prime $p>2$ (resp. at
$p=2$) such that $\pi_0 R$ contains a primitive $p$th root of unity (resp. a
$4$th root of unity).

One of the primary motivations for desiring a positive answer to Question
\ref{q1} is because it often allows for the construction of sheaves of
$\Eoo$-rings on chromatically interesting moduli stacks $M$ (such as the moduli
stack of formal groups $\Mfg$ and the moduli stack of elliptic curves) which
lift the structure sheaf of $M$. This is the perspective through which Question
\ref{q1} is discussed in \cite{goerss-realizing}. For instance, Theorem
\ref{affirmative1} was used by Goerss--Hopkins--Miller to construct a sheaf of
even-periodic $\Eoo$-rings on the \'etale site of the moduli stack of elliptic
curves; the global sections of this sheaf is the $\Eoo$-ring $\TMF$. A lot of
large-scale phenomena in chromatic homotopy theory are shadows of the hope
(discussed in \cite{goerss-realizing}) that there is some Grothendieck topology
on the moduli stack $\Mfg$ of one-dimensional formal groups for which the
structure sheaf $\co_{\Mfg}$ admits a lift to a sheaf of $\Eoo$-rings; the
global sections of this $\Eoo$-ring would be the sphere spectrum. Theorem
\ref{failure2} places severe restrictions on any Grothendieck topology on any
moduli stack $M$ (including $\Mfg$) for which there is a sheaf of $\Eoo$-rings
on the associated site lifting the structure sheaf of $M$.

Our proof of Theorem \ref{failure2} spans the whole of Section \ref{mainproofs},
and relies heavily on power operations obtained from the $H_\infty$-ring
structure imposed on $R$. The proof also relies on a result of Hahn's from
\cite{hahn}, which says that a $K(n)$-acyclic $H_\infty$-ring is
$K(n+1)$-acyclic.

The following is a natural question motivated by Theorem \ref{failure2}:
\begin{question}\label{higher-chrom}
    Consider the setup of Question \ref{q1}, and let $k\geq 1$ be finite. Is
    there an analogue of Theorem \ref{failure2} for $K(n)$-local $\E{k}$-rings?
    Namely, is there a $K(n)$-local $\E{k}$-ring $R$ such that $\pi_0 R$
    contains a primitive $p$th root of unity?
\end{question}
\begin{remark}
    Other notions of ramification in the world of structured ring spectra have
    been studied. For instance, the map $\bo\to \bu$ is an example of a
    ramified extension in the sense of Dundas--Lindenstrauss--Richter
    \cite{ramified-ring-spectra}.
\end{remark}

In Section \ref{lt-tower}, we study an application of Theorem \ref{failure2} to
lifting the Lubin--Tate tower from arithmetic geometry into the realm of
spectral algebraic geometry. The Lubin--Tate tower, originally introduced in
\cite{drinfeld}, and further studied (for instance) in \cite[Chapter 3]{rz},
plays an important role in the proof of the Jacquet--Langlands correspondence
(see \cite{rogawski}), and also carries information about power operations on
Morava $E$-theory by a result of Strickland's (see \cite{strickland-subgp}). As
the name suggests, it is a tower of moduli problems (which are in fact
represented by affine formal schemes) living over the Lubin--Tate moduli space.
The $n$th level of this tower parametrizes deformations of formal groups along
with a level $\Gamma_1(p^n)$-structure, i.e., a basis for the $p^n$-torsion in
the deformed formal group. As a consequence of Theorem \ref{failure2}, we prove:
\begin{theorem}\label{failure3}
    Let $\M_k^h = \spf A_k^h$ denote the Lubin--Tate space at level $p^k$ and
    height $h$, and let $E_h$ denote Morava $E$-theory at height $h$ and the
    prime $p$. If $k>0$, there is no $H_\infty$-ring over $E_h$ whose
    underlying ring is $A^h_k$.
\end{theorem}
It follows from the main result of \cite{strickland-subgp} that there is an
isomorphism $A^h_k \cong E_h^0(BC_{p^k})/\tr$, where $\tr$ is a certain ideal of
$E_h^0(BC_{p^k})$, known as the transfer ideal. The ring $E_h^0(BC_{p^k})/\tr$
parametrizes power operations on height $h$ Morava $E$-theory. Consequently,
Theorem \ref{failure3} says that rings of power operations do not lift to
derived algebraic geometry. A different proof of Theorem \ref{failure3} was
provided by Niko Naumann in private communication.

\subsection{Acknowledgements}
I'd like to thank Tyler Lawson for suggesting that a result like Theorem
\ref{failure2} is true, and for reading over an earlier draft of this paper.
I'm also grateful to Matt Ando, Jeremy Hahn, Marc Hoyois, Alex Mennen,
Catherine Ray, and Paul VanKoughnett for useful conversations, as well as to an
anonymous reader for a careful reading of this paper and helpful suggestions on
how to improve it.

\section{Ramification and structured ring spectra}\label{mainproofs}   
We will use the following terminology.
\begin{definition}
    A $\theta$-ring is a ring $A$ with an operator $\theta\colon A\to A$ such
    that the following identities are satisfied:
    \begin{align*}
	\theta(0) & = 0\\
	\theta(x+y) & = \theta(x) + \theta(y) + \frac{x^p+y^p-(x+y)^p}{p}\\
	\theta(xy) & = \theta(x) y^p + \theta(y) x^p + p\theta(x) \theta(y).
    \end{align*}
\end{definition}
These conditions imply that $\psi^p(x) = x^p + p\theta(x)$ is a ring
endomorphism of $A$.

The following result is described in \cite[Chapter VIII]{hinfty},
\cite{bousfield-power-ops}, and \cite[Section 3]{k1local} (the latter for
$\Eoo$-rings).
\begin{theorem}
    Let $R$ be a $K(1)$-local $H_\infty$-ring. Then, $\pi_0 R$ naturally has
    the structure of a $\theta$-ring via power operations, denoted $\psi^p$ and
    $\theta$.
\end{theorem}
Using the existence of these power operations, Mike Hopkins (in unpublished
work) gave a proof of the following fact.
\begin{lemma}\label{mjh}
    Let $R$ be the spectrum $K_p[\zeta_p]$ obtained by adjoining a primitive
    $p$th root of unity to $p$-adic $K$-theory.  Then $R$ is not an
    $H_\infty$-ring.
\end{lemma}
\begin{proof}
    The spectrum $R$ is $K(1)$-local.  If $\pi_0 R$ contained a primitive $p$th
    root of unity $\zeta$, then the power operations would supply an equality
    $$\psi^p(\zeta) = 1 + p\theta(\zeta).$$
    Since $\psi^p$ is a ring homomorphism, if $\Phi_p$ denotes the $p$th
    cyclotomic polynomial, we have
    $$\Phi_p(\psi^p(\zeta)) = \psi^p(\Phi_p(\zeta)) = \psi^p(0) = 0,$$
    so $\psi^p(\zeta)$ is another primitive $p$th root of unity, i.e.,
    $\psi^p(\zeta) = \zeta^k$ for some $1\leq k\leq p-1$.  Since
    $\theta(\zeta)\neq 0$, the prime $p$ must divide $1-\zeta^k$, which is
    impossible.
\end{proof}
In an email, Tyler Lawson gave a proof of the following generalization of
Lemma \ref{mjh} at the prime $2$.
\begin{prop}\label{tyler}
    There is no $K(1)$-local $H_\infty$-ring $R$ such that $\pi_0 R$ contains a
    primitive $4$th root of unity.
\end{prop}
Lawson asked if this result can be generalized to all heights and primes.
Theorem \ref{failure2} gives an affirmative answer to this question.
\begin{proof}
    Denote by $i$ the fourth root of unity in $\pi_0 R$. One can easily check
    that
    $$\theta(xy) = \theta(x)y^2 + \theta(y)x^2 + 2\theta(x)\theta(y);$$
    therefore, if $x=y=i$, then
    $$-1 = \theta(-1) = -2(\theta(i) - \theta(i)^2).$$
    This, however, implies that $2$ is invertible, which is impossible by Lemma
    \ref{p-inv}.
\end{proof}
In the proof of Proposition \ref{tyler}, we utilized the following easy
result.
\begin{lemma}\label{p-inv}
    Let $E$ be a $p$-local ring spectrum (where $p$ is any prime). If $p$ is a
    unit in $\pi_0 E$, then $E$ is $K(n)$-locally trivial for any $n\geq 1$.
\end{lemma}
\begin{proof}
    Since $p$ is a unit in $\pi_0 E$, the spectrum $E$ is $p$-locally rational.
    This implies that $K(n)_\ast E$ is zero for any $n\geq 1$.
\end{proof}
\begin{prop}\label{mainthm}
    Let $p>2$ be an odd prime. Suppose $A$ is a $\theta$-ring such that $p$ is
    not a unit in $A$. Then $A$ does not contain a primitive $p$th root of
    unity $\zeta$.
\end{prop}
\begin{proof}
    By definition, we have
    $$\theta(x+y) = \theta(x) + \theta(y) +
    \sum^{p-1}_{i=1}\frac{1}{p}\binom{p}{i} x^i y^{p-i}.$$
    By induction, we find that
    \begin{align}
	\theta\left(\sum^{k}_{i=1}x_i\right) & = \sum^k_{i=1} \theta(x_i) +
	\frac{\sum_{i=1}^k x_i^p - \left(\sum_{i=1}^k x_i\right)^p}{p}
	\label{multsum}
    \end{align}

    Suppose $A$ is a $\theta$-ring containing a primitive $p$th root of unity.
    Since
    \begin{align}\label{sump}
	1+\zeta+\cdots+\zeta^{p-1} = 0,
    \end{align}
    we learn that
    $$\theta(1+\zeta+\cdots+\zeta^{p-1}) = 0.$$
    Applying Equation \eqref{multsum} to the left hand side, we get
    \begin{align}\label{toprove}
	0 & = \sum^{p-1}_{i=0}\theta(\zeta^i) + \frac{\sum_{i=0}^{p-1}
	\zeta^{ip} - \left(\sum_{i=0}^{p-1} \zeta^i\right)^p}{p}
    \end{align}
    Equation \eqref{sump} implies that 
    \begin{align}\label{secondterm}
	\frac{\sum_{i=0}^{p-1} \zeta^{ip} - \left(\sum_{i=0}^{p-1}
	\zeta^i\right)^p}{p} = \frac{p-0}{p} = 1.
    \end{align}
    It remains to compute $\sum^{p-1}_{i=0}\theta(\zeta^i)$. As $\psi^p$ is
    multiplicative, we deduce that
    $$\theta(x^n) = \frac{(x^p+p\theta(x))^n-x^{np}}{p};$$
    this implies that the first term on the right hand side of Equation
    \eqref{toprove} is
    $$\sum^{p-1}_{i=0}\theta(\zeta^i) =
    \sum^{p-1}_{i=1}\frac{(1+p\theta(\zeta))^i-1}{p} =
    \sum^{p-1}_{i=1}\sum^i_{k=1} \binom{i}{k} p^{k-1}\theta(\zeta)^k.$$
    Splitting up terms, this sum becomes
    \begin{align*}
	\sum^{p-1}_{i=0}\theta(\zeta^i) & = \theta(\zeta)\sum_{i=1}^{p-1}i +
	\sum^{p-1}_{i=1}\sum^i_{k=2} \binom{i}{k} p^{k-1}\theta(\zeta)^k \\
	& = p\left(\frac{p-1}{2}\theta(\zeta) + \sum^{p-1}_{i=1}\sum^i_{k=2}
	\binom{i}{k} p^{k-2}\theta(\zeta)^k\right).
    \end{align*}
    Let
    $$Z = \frac{p-1}{2}\theta(\zeta) + \sum^{p-1}_{i=1}\sum^i_{k=2}
    \binom{i}{k} p^{k-2}\theta(\zeta)^k.$$
    Equation \eqref{toprove} and Equation \eqref{secondterm} show that
    $$0 = pZ + 1,$$
    which implies that $p$ is invertible in $A$; contradiction.
\end{proof}
Before proceeding, let us recall a theorem of Hahn's from \cite{hahn}:
\begin{theorem}[Hahn]\label{jeremy}
    A $K(n)$-acyclic $H_\infty$-ring $R$ is $K(n+1)$-acyclic.
\end{theorem}
\begin{proof}[Proof of Theorem \ref{failure2}]
    Proposition \ref{tyler} proves the result when $p=2$, so we can restrict to
    the case when $p$ is odd. Suppose $R$ is a nontrivial $K(1)$-local ring
    such that $\pi_0 R$ contains a primitive $p$th root of unity. By Lemma
    \ref{p-inv}, $p$ is not a unit in $\pi_\ast R$. If $n=1$, then $\pi_0 R$ is
    a $\theta$-ring. It follows from Proposition \ref{mainthm} that $\pi_0 R$
    is trivial, so $R$ is contractible.

    To prove the corollary at all heights, suppose $n\geq 2$, and let $R$ be a
    $K(n)$-local $\Eoo$-ring such that $\pi_0 R$ contains a primitive $p$th
    root of unity $\zeta_{p}$. As $\pi_0 R\xar{\pi_0\Lone} \pi_0 \Lone R$ is a
    ring homomorphism, the element $\pi_0\Lone(\zeta_{p})$ is another primitive
    $p$th root of unity in $\pi_0 \Lone R$.  By the arguments laid out above,
    $\Lone R\simeq 0$, i.e., $R$ is $K(1)$-acyclic. Theorem \ref{jeremy}
    therefore implies that $R$ is $K(n)$-acyclic for every $n\geq 2$, so
    $R\simeq \Lk R\simeq 0$, as desired.
\end{proof}
\begin{remark}
    By Theorem \ref{failure2}, any $\EOo$-ring $R$ such that $p^{1/p}\in \pi_0
    R$ must be $K(n)$-acyclic for every $n\geq 1$, since $p^{1/p} + 1$ is, up to
    units, a primitive $p$th root of unity. This fact can also be proven
    directly, by applying $\theta(-)$ to the equation $(p^{1/p})^p - p = 0$ and
    using the multiplicativity of $\psi^p$ as in the proof above.
\end{remark}

\section{The Lubin--Tate tower}\label{lt-tower}
In arithmetic geometry, the \emph{Lubin--Tate tower} (originally introduced in
\cite{drinfeld}; see also \cite[Chapter 3]{rz} and \cite[Section 2]{weinstein})
is a tower of finite flat extensions of the Lubin--Tate space, given by the
rings representing deformations (of the Honda formal group) along with a chosen
Drinfel'd level structure. These level structures play an important role in
constructing power operations on Morava $E$-theory. They are also of immense
arithmetic interest, as the Lubin--Tate tower simultaneously admits actions of
three groups which form the main object of study in local Langlands.

Let $R$ be a complete local Artinian $\W(\ol{\FF_p})$-algebra, and let
$(G,\iota)$ be a deformation of the Honda formal group of height $h$ to $R$.
Let $\fr{m}$ denote the maximal ideal of $R$ equipped with the group law given
by $G$.
\begin{definition}
    A \emph{Drinfel'd level $p^k$-structure} is a homomorphism $\phi\colon
    (\Z/p^k)^h\to \fr{m}$ such that there is an inequality of Cartier divisors
    $$\sum_{a\in (\Z/p^k)^h[p]} [\phi(a)] \leq G[p].$$
    If we pick a coordinate $x$ on $G$, then this amounts to asking that
    $$\left(\prod_{a\in (p^{k-1}\Z/p^k\Z)^h} (x-\phi(a))\right)|[p](x)$$
    inside $R[\![x]\!]$.
\end{definition}
Drinfel'd proved the following result (see \cite[Proposition 4.3]{drinfeld}).
\begin{theorem}[Drinfel'd]\label{drinfeld-rep}
    The functor that sends a complete local Artinian $\W(\ol{\FF_p})$-algebra
    $(R, \fr{m})$ to the set of deformations of the Honda formal group (of
    height $h$) along with a level $p^k$-structure is corepresented by an
    $\pi_0 E_h$-algebra $A_k^h$.
\end{theorem}
\begin{remark}\label{construct-drinfeld}
    These rings can be defined inductively, by adjoining the $p^k$-torsion in
    $G$; for instance, if $k>1$, then $A_k^h = A^h_{k-1}[x_1, ...,
    x_h]/([p](x_i) - \phi(a_i)),$ where $\{a_i\}$ is the canonical basis for
    $(\Z/p^k)^n$.
\end{remark}
\begin{example}\label{h-1-example}
    At height $h=1$, a level $p^k$-structure on $\GG_m$ is a choice of basis for
    the $p^k$-torsion; this is exactly the choice of a primitive $p^k$-torsion
    point. One might therefore expect that $A^1_k \cong \Z_p[\zeta_{p^k}]$. This
    is indeed true: we have $A_1^1 \cong \Z_p[\zeta_p]$ (see, e.g.,
    \cite[Section 2.3]{weinstein}; alternatively, this follows immediately from
    the explicit description given in \cite[Lemma in Proposition
    4.3]{drinfeld}). It follows from this that if $k>1$, then $A_k^1 \cong
    \Z_p[\zeta_{p^k}]$. Indeed, recall that $[p](x) = (1+x)^p-1$.  Moreover,
    note that if $\phi\colon (\Z/p^{k-1})^h\to A_{k-1}^1$ is the universal level
    $p^{k-1}$-level structure, then $\phi(1) = \zeta_{p^{k-1}}-1$. Using Remark
    \ref{construct-drinfeld}, we find that $A_k^1 \cong A_{k-1}^h[x]/([p](x) -
    \phi(1)) \cong \Z_p[\zeta_{p^k}],$ as desired.
\end{example}
The main result of this section is:
\begin{theorem}\label{towernolift}
    Let $\M_k^h = \spf A_k^h$ denote the Lubin--Tate space at level $p^k$ and
    height $h$, and let $E_h$ denote Morava $E$-theory at height $h$ and the
    prime $p$. If $k>0$, there is no $H_\infty$-ring over $E_h$ whose
    underlying ring is $A^h_k$.
\end{theorem}
\begin{proof}
    There is a ``determinant'' map of formal schemes $\M_k^h \to \M_k^1$
    constructed in \cite[Section 2.5]{weinstein}. We showed in Example
    \ref{h-1-example} that $A_k^1 \cong \Z_p[\zeta_{p^k}]$. The image of
    $\zeta_{p^k}$ under the induced map $\Z_p[\zeta_{p^k}] \to A_k^h$ is a
    primitive $p^k$th root of unity in $A_k^h$. Applying Theorem \ref{failure2},
    we conclude that $A_k^h$ does not lift to a $K(h)$-local $H_\infty$-ring; in
    particular, there is no $K(h)$-local $E_h$-algebra whose homotopy is
    isomorphic to $A_k^h$.
\end{proof}
\begin{remark}
    The PEL Shimura varieties with Drinfel'd level structures considered in
    \cite[\S III.1]{harris-taylor} do not lift to derived stacks: if they did,
    then the completion at a height $h$ point would give an $\Eoo$-ring lifting
    $A_k^h$, contradicting Theorem \ref{towernolift}.
\end{remark}

\bibliographystyle{amsalpha}
\bibliography{main}
\end{document}